\documentclass[12pt,reqno]{amsart}
\usepackage{amsthm}
\usepackage{amsfonts}
\usepackage{amsmath}
\usepackage{amssymb}
\usepackage{mathrsfs}
\usepackage{epsfig}
\usepackage{graphicx}
\usepackage{epstopdf}
\usepackage{tikz-cd}
\usepackage{color}
\usepackage{ifthen}
\usepackage{float}
\usepackage{cancel} 
\usepackage{comment}
\usepackage{stmaryrd}
\usepackage{hyperref}
\makeatletter
\@namedef{subjclassname@2010}{%
\textup{2010} Mathematics Subject Classification}
\makeatother
\usepackage[normalem]{ulem}
\newtheorem{theorem}{Theorem}[section]
\newtheorem{proposition}[theorem]{Proposition}
\newtheorem{corollary}[theorem]{Corollary}
\newtheorem{lemma}[theorem]{Lemma}

\newtheorem{remark}[theorem]{\bf Remark}

\theoremstyle{remark}

\theoremstyle{definition}

\newcommand{\bq}{\begin{equation}}
\newcommand{\eq}{\end{equation}}
\newcommand{\beqn}{\begin{eqnarray*}}
\newcommand{\eeqn}{\end{eqnarray*}}
\newcommand{\beq}{\begin{eqnarray}}
\newcommand{\eeq}{\end{eqnarray}}

\usepackage{soul}

\newcommand{\bc}{\begin{centre}}
\newcommand{\ec}{\end{centre}}

\newcommand{\dist}{{\rm\,dist}}

\newcommand{\sgn}{{\rm sgn\,}}
\newcommand{\ba}{\begin{array}}
\newcommand{\ea}{\end{array}}

\newcommand{\inp}[2]{\langle{#1},\,{#2} \rangle}

\renewcommand{\Delta}{{\nabla}}

\def \z{\boldsymbol{z}}
\def \\theta{\boldsymbol{\theta}}
\def \\xi{\boldsymbol{\xi}}

\begin{document}
\title[Commutant lifting]{Commutant lifting and interpolation on quotients of bounded symmetric domains}

\author[M. K. Mal]{Milan Kumar Mal}
\address[M. K. Mal]{Department of Mathematics, Indian Institute of Technology Madras, Chennai 600036, India}
 \email{ma21d018@smail.iitm.ac.in; milanmal1702@gmail.com }

\thanks{Support for the work of the author was provided in the form of a Prime Minister's Research Fellowship (PMRF / 2502827).}




\date{}

\begin{abstract}
Let $\Omega\subseteq \mathbb C^d$ be a bounded symmetric domain, $G$ a finite complex reflection group acting on $\mathbb C^d$, and $\boldsymbol \theta:\Omega\to \boldsymbol \theta(\Omega)$ the associated proper holomorphic map factored by $G.$ In this paper, we investigate commutant lifting and interpolation by Schur functions on the quotient domain $\boldsymbol \theta(\Omega).$
For a given quotient module of the Hardy space $H^2(\boldsymbol\theta(\Omega))$, 
we obtain equivalent criteria for a contractive module map to admit a Schur-class lift: one in terms of the contractivity of an associated functional on a subspace of $L^1(\partial\boldsymbol\theta(\Omega))$, and another in terms of a geometric distance formula in the same $L^1$-space.
Specializing to quotient domains of the polydisc factored by imprimitive finite complex reflection groups, we obtain a commutant lifting criterion formulated in terms of inner functions.
Finally, we apply these operator-theoretic results to finite-point Nevanlinna--Pick type interpolation problems on $\boldsymbol \theta(\Omega)$.  
Since the symmetrized bidisc and the tetrablock arise as quotient domains of suitable bounded symmetric domains, these criteria apply in particular to those domains.

\end{abstract}


\maketitle

\section{Introduction}
Interpolation by bounded analytic functions on the unit disc $\mathbb D \subseteq \mathbb C$ has long been a central topic in complex analysis and operator theory. Given distinct points $z_1,\ldots, z_p\in \mathbb D$ and scalars $w_1,\ldots, w_p \in \mathbb D,$ the classical Nevanlinna-Pick interpolation problem asks for necessary and sufficient condition for the existence of a holomorphic function $\varphi: \mathbb D \to \mathbb C$ satisfying $$\sup_{z\in \mathbb D} |\varphi(z)|\leq 1, \;\;\text{and}\;\;\varphi(z_i)=w_i\;\; \text{for all}\;\; i=1,\ldots, p.$$
The celebrated theorem of Nevanlinna and independently Pick asserts that such a function exists if and only if the associated $p\times p$ matrix, called the {\it  Pick matrix}, given by 
\beq \label{pick matrix} \left(\frac{1-w_i \bar{w}_j}{1-z_i \bar{z}_j}
\right)_{i,j=1}^p,\eeq is positive semi-definite.  

After nearly four decades, Donald Sarason introduced an operator-theoretic approach to the interpolation problem in \cite{DS1967}. His result, now known as the commutant lifting theorem, may be formulated as follows. Let $H^2(\mathbb D)$ be the Hardy space on $\mathbb D$, and let $T_{z}$ denote the multiplication operator on $H^2(\mathbb D)$ by the coordinate function $z$. A closed subspace $\mathcal{Q}$ of $H^2(\mathbb D)$ is called a quotient module if $T_{z}^* \mathcal{Q}\subseteq \mathcal{Q}.$ Let $P_{\mathcal{Q}}$ be the orthogonal projection of $H^2(\mathbb D)$ onto $\mathcal{Q}.$ 
Sarason’s theorem states that if $X$ is a bounded linear operator on $H^2(\mathbb D)$ satisfying $X P_{\mathcal{Q}} T_{z}|_{\mathcal{Q}}= P_{\mathcal{Q}} T_z |_{\mathcal{Q}}X$, then there exists a bounded holomorphic function $\varphi$ on $\mathbb D$ such that $$X=P_{\mathcal{Q}} T_{\varphi}|_{\mathcal{Q}} \;\; \text{and}\;\; \|X\|=\|\varphi\|_{\infty},$$ where $T_{\varphi}$ denotes the multiplication operator by $\varphi$ on $H^2(\mathbb D).$ As an application, Sarason recovered the Nevanlinna–Pick interpolation theorem.

The multivariable counterparts of both the interpolation and commutant lifting have since attracted considerable attention. In contrast to the disc, interpolation on higher dimensional domains is generally not governed by a single Pick-type matrix. The interpolation problem on the bidisc was solved by Agler and McCarthy in \cite{AM1999}. The same problem on the symmetrized bidisc $\mathbb G_2$ has been resolved recently in \cite{AY2017, BS2018}. 
There have also been several efforts to address interpolation and commutant lifting problems in general multivariable settings; see, for instance, \cite{BT1998, BLTT1999, BB2004}. However, many of these results apply only to restricted classes of functions or operators. More recently, a comprehensive characterization of both interpolation and commutant lifting for the polydisc $\mathbb D^d$
and the Euclidean unit ball $\mathbb B_d$ have been obtained in \cite{DeS2026, BDS2025}. 

The purpose of this paper is to develop commutant lifting and interpolation results on a broad class of domains that arise as the image of a bounded symmetric domain under a proper holomorphic map factored by a finite complex reflection group. 

Recall that a bounded domain $\Omega\subseteq \mathbb C^d$ is {\it symmetric} if, for every pair of points $\z, \z'\in \Omega,$ there exists an involution in the group of biholomorphic automorphisms $\mathrm{Aut}(\Omega)$ interchanging them. 
In their standard realization, bounded symmetric domains are circled, convex, and contain the origin. The open unit disc $\mathbb D$, the unit polydisc $\mathbb D^d$, and the Euclidean ball $\mathbb B_d$ in $\mathbb C^d$ are some well studied examples of bounded symmetric domains. For a masterful exposition on bounded symmetric domains, the reader is referred to \cite{LO1977, AJ1995}.

Given bounded domains $\Omega_1,\Omega_2\subseteq \mathbb C^d$, a holomorphic map $\pi: \Omega_1\to \Omega_2$ is said to be {\it proper} if the preimage of every compact subset of $\Omega_2$ is compact in $\Omega_1.$ 
Such a map is necessarily surjective, and it follows from \cite[Section 1.2]{RD1982} that there exists a positive integer $n$ (referred to as {\it multiplicity} of $\pi$) such that
$$\text{cardinality of } \pi^{-1}(\boldsymbol w) \begin{cases}
    =n, \qquad \boldsymbol w\in \Omega_2\setminus\pi(N(J_\pi))\; \text{and}\\
     < n, \qquad \boldsymbol w\in \pi(N(J_\pi)),
\end{cases}$$
where $J_\pi$ is the complex Jacobian determinant of $\pi$, and  $N(J_\pi)$ denotes its zero set.
We say $\pi$ is factored by {\it automorphisms} if there exists a finite subgroup $G\subseteq \mathrm{Aut}(\Omega_1)$ such that $\pi^{-1}(\pi(\z))=\{\sigma(\z): \sigma\in G\}$ for all $\z\in \Omega_1.$ 
It is known that such a group $G$ is either a group generated by complex reflections or conjugate to a complex reflection group, where a {\it complex reflection} is a finite order linear map $\sigma$ on $\mathbb C^d$ such that the rank of $I_d-\sigma$ is $1$. A group $G$ generated by finite number of complex reflections is referred as {\it finite complex reflection group}.


A bounded domain $\Omega$ in $\mathbb C^d$ is said to be $G$-invariant if $\sigma\cdot\z:=\sigma^{-1}(\z)\in \Omega$ for all $\sigma\in G.$ 
This induces an action on function $f:\Omega \to \mathbb C$ defined by  $(\sigma\cdot f)(\z):=f(\sigma^{-1}\cdot \z)$, and we say $f$ is {\it $G$-invariant} if $\sigma \cdot f=f$ for all $\sigma\in G.$ A fundamental characterization of finite complex reflection groups is that the ring of $G$-invariant polynomials is generated by a set of $d$ algebraically independent homogeneous polynomials (known as the {\it basic polynomials}) $\{\theta_1,\ldots, \theta_d\}$ associated to $G$ \cite[p. 282]{ST1954}. The mapping $\boldsymbol \theta=(\theta_1,\ldots, \theta_d): \Omega \to \boldsymbol \theta(\Omega)$ is proper holomorphic and factored by $G$ \cite{RD1982, TM2013}. Any other proper holomorphic map factored by $G$ factors through $\boldsymbol \theta$ via a biholomorphism.
Consequently, we consider $\boldsymbol \theta(\Omega)$ as the image of $\Omega$ under a proper holomorphic map factored by $G.$ 
Although the map $\boldsymbol \theta$ is not unique, the degrees of $\theta_j$'s are unique for $G$ up to order. Throughout this paper, $\Omega$ denotes a bounded symmetric domain, $G$ a finite complex reflection group acting on $\mathbb C^d$ under which $\Omega$ is invariant, and $\boldsymbol \theta(\Omega)$ the associated {\it quotient domain}, with typical elements denoted by $\boldsymbol \zeta=(\zeta_1,\ldots, \zeta_d)=\boldsymbol \theta(\z).$

Quotient domains have become an active subject of research in function theory and operator theory; see, for instance, \cite{BS1982, TM2013, MRZ2013, TirDS2021,  DS2021, BDGR2022,  GN2023, BBS2024}, as well as \cite{GR2024} and the references therein.
A suitable framework for Hardy spaces on quotient domains, together with the associated theory of Toeplitz operators, was developed in \cite{GR2024} (see also \cite{MRZ2013}). 
For our convenience, we recall in Section \ref{S2} those aspects of the Hardy space construction that are relevant to the present work.

Let $H^2(\boldsymbol \theta(\Omega))$ denote the Hardy space on the quotient domain $\boldsymbol \theta(\Omega),$ and let $\boldsymbol T_{\boldsymbol \zeta}=(T_{\zeta_1}, \ldots, T_{\zeta_d})$ be the commuting tuple of multiplication operators on $H^2(\boldsymbol \theta(\Omega))$ induced by the coordinate functions. 
We denote by $H^\infty(\boldsymbol \theta(\Omega))$ the Banach algebra of bounded holomorphic functions on $\boldsymbol \theta(\Omega),$ endowed with the supremum norm $\|\cdot\|_{\infty, \boldsymbol \theta(\Omega)}.$ 
The elements of the closed unit ball $\mathcal{S}(\boldsymbol{\theta}(\Omega))$ of $H^\infty(\boldsymbol{\theta}(\Omega))$ are referred to as {\it Schur functions}.
For any $\varphi \in H^\infty(\boldsymbol \theta(\Omega))$, the analytic Toeplitz operator $T_{\varphi}$ on $H^2(\boldsymbol \theta(\Omega))$ is the multiplication operator $$T_{\varphi}:H^2(\boldsymbol \theta(\Omega)) \to H^2(\boldsymbol \theta(\Omega)), \;\; T_\varphi f=\varphi f.$$ Moreover, $\|T_{\varphi}\|=\|\varphi\|_{\infty, \boldsymbol \theta(\Omega)}:=\sup\{|\varphi(\boldsymbol \zeta)|: \boldsymbol \zeta\in \boldsymbol \theta(\Omega)\}.$
Let $\mathcal{Q}$ be a quotient module of $H^2(\boldsymbol \theta(\Omega)),$ that is, a closed subspace invariant under $T_{\zeta_j}^*$ for each $ j=1,\ldots,d.$ Let $P_{\mathcal{Q}}$ denote the orthogonal projection of $H^2(\boldsymbol \theta(\Omega))$ onto $\mathcal{Q}.$ For any $\varphi\in H^\infty(\boldsymbol \theta(\Omega))$, the compression $S_{\varphi}:=P_{\mathcal{Q}} T_{\varphi}|_{\mathcal{Q}}$ defines a bounded operator on $\mathcal{Q}.$ 
In particular, $S_{\zeta_j}=P_{\mathcal{Q}} T_{\zeta_j}|_{\mathcal{Q}}\in \mathcal{B}(\mathcal{Q})$ for all $j=1,\ldots,d.$ It readily follows that $S_{\varphi }$ commutes with $S_{\zeta_j}$ for all $j,$ and that $\|S_{\varphi}\| \leq 1$ whenever $\varphi \in \mathcal{S}(\boldsymbol \theta(\Omega)).$
Motivated by this observation, we define a {\it module map} as any operator $X\in \mathcal{B}(\mathcal{Q})$ satisfying $XS_{\zeta_j}=S_{\zeta_j}X$ for all $j=1,\ldots,d.$ Thus every Schur function $\varphi\in \mathcal{S}(\boldsymbol \theta(\Omega))$ induces a contractive module map $S_{\varphi}$ on $\mathcal{Q}$.
Let $\mathcal{B}_1(\mathcal{Q})=\{X \in \mathcal{B}(\mathcal{Q}): \|X\|\leq 1\}.$ We say that a contractive module map $X\in \mathcal{B}_1(\mathcal{Q})$ is {\it liftable} or {\it admits a lift} if there exists a function $\varphi \in \mathcal{S}(\boldsymbol \theta(\Omega))$ such that $X=S_\varphi.$

The main objective of this paper is twofold. First, we establish several equivalent characterizations for the existence of Schur class lifts of contractive module maps on quotient modules of $H^2(\boldsymbol \theta(\Omega))$. These results are developed in Section \ref{S3}. In the special case of the quotient domain $\boldsymbol \theta(\mathbb D^d)$, arising from the polydisc $\mathbb D^d,$ we obtain a commutant lifting criterion expressed in terms of inner functions on $\boldsymbol \theta(\mathbb D^d)$ (see Theorem \ref{3rd characterization}). 

Secondly, in Section \ref{S4}, we apply these commutant lifting characterizations to interpolation problems. More precisely, for prescribed distinct points in $\boldsymbol \theta(\Omega)$ and prescribed target values in the unit disc, we derive criteria for the existence of a Schur class interpolant in $\mathcal{S}(\boldsymbol \theta(\Omega)).$ 
The symmetrized bidisc $\mathbb G_2$ and the tetrablock $\mathbb E$, which arise in connection with $\mu$-synthesis, are notable examples of quotient domains that have attracted considerable interest in geometric function theory and operator theory. The general results obtained in this paper apply in particular to these domains and therefore yield interpolation criteria on $\mathbb G_2$ and $\mathbb E.$

\section{Preliminaries}\label{S2}
In this section, we recall the Hardy space on the quotient domain $\boldsymbol \theta(\Omega)$ and describe its boundary realization. Recall that $\boldsymbol \theta=(\theta_1,\ldots, \theta_d):\Omega \to \boldsymbol \theta(\Omega)$ is the proper holomorphic map factored by the finite complex reflection group $G.$ We write $\partial\Omega$ and  $\partial\boldsymbol \theta(\Omega)$ for the Shilov boundaries of $\Omega$ and $\boldsymbol \theta(\Omega)$, respectively.
Since the quotient map $\boldsymbol\theta$ extends as a proper holomorphic map of the same multiplicity to suitable neighborhoods of the corresponding closures, it follows from \cite[p. 100, Corollary 3.2]{KZ2013} that $\partial\boldsymbol \theta(\Omega)=\boldsymbol \theta(\partial\Omega).$ 


Let $\mathbb K=\{\phi\in \mathrm{Aut}(\Omega): \phi(0)=0\}$ be the isotropy subgroup at the origin, and let $\mu$ denote the normalized $\mathbb K$-invariant Lebesgue measure on $\partial\Omega.$ Then the Hardy space $H^2(\Omega)$ consists of all holomorphic functions $f:\Omega\to \mathbb C$ such that  
$$\|f\|_{H^2(\Omega)}^2=\sup_{0<r<1} \int_{\partial\Omega} |f(r\z)|^2d\mu(z) <\infty.$$
Let $L^2(\partial\Omega)$ denote the space of square-integrable functions with respect to $\mu.$ By \cite[Theorem 6]{HM1969}, every $f\in H^2(\Omega)$ has a boundary value function $f^* \in L^2(\partial\Omega)$ satisfying $$\lim_{r \to 1^-} \int_{\partial\Omega}|f_r-f^*|^2 d\mu =0,\quad \text{and} \quad \|f\|_{H^2(\Omega)}=\|f^*\|_{L^2(\partial\Omega)},$$ where the radial slice  $f_r:\partial\Omega \to \mathbb C$ is given by $f_r(\z)=f(r\z)$ for $0<r<1$ (see also \cite[p. 126]{UH1996}). Thus, the boundary value map $i: H^2(\Omega) \to L^2(\partial\Omega)$, defined by $i(f)=f^*$, embeds $H^2(\Omega)$ isometrically into $L^2(\partial\Omega)$. We denote its range by $H^2(\partial\Omega)$. 

The space $H^2(\Omega)$ is a reproducing kernel Hilbert space. Its reproducing kernel, denoted by $\mathbb S_\Omega$, is the Szeg\"o kernel of $\Omega.$
The corresponding Poisson kernel $P:\Omega \times \partial\Omega  \to \mathbb C$ is given by $$P( \z, \boldsymbol \xi)=\frac{|\mathbb S_\Omega(\z, \boldsymbol \xi)|^2}{\mathbb S_\Omega(\z,\z)}, \qquad \z \in \Omega,\; \boldsymbol \xi \in \partial\Omega.$$ By \cite[Theorem 1]{HM1973} (see also \cite{KA1965}), any $f\in H^2(\Omega)$ can be recovered from its boundary value $f^* \in L^2(\partial\Omega)$ via the Poisson integral formula:
$$f(\z)=\int_{\partial\Omega}  P(\z, \boldsymbol \xi) f^*(\boldsymbol \xi)d\mu(\boldsymbol \xi), \qquad \z \in \Omega.$$
Moreover, $f_r$ converges to $f^*$ almost everywhere on $\partial\Omega$ (see \cite[Theorem 2]{WN1968}), that is,
$$f^*(\z)=\lim_{r \to 1^-} f(r\z)\qquad \text{for a.e}\; \z \in \partial \Omega.$$  

We now recall the Hardy space on the quotient domain. Using the change of variables formula, define a measure $\mu_{\boldsymbol\theta}$ supported on $\partial \boldsymbol \theta(\Omega) $ by
$$\int_{\partial\boldsymbol{\theta} (\Omega)} f d\mu_{\boldsymbol\theta}= \int_{\partial\Omega} (f\circ \boldsymbol\theta) |J_{\boldsymbol \theta}|^2 d\mu,$$ 
where $J_{\boldsymbol\theta}$ denotes the complex Jacobian determinant of the proper holomorphic map $\boldsymbol\theta: \Omega \to \boldsymbol \theta (\Omega).$ The Hardy space $H^2(\boldsymbol \theta(\Omega))$  is then defined as the Hilbert space of holomorphic functions $f$ on $\boldsymbol \theta(\Omega)$ for which $J_{\boldsymbol \theta} (f\circ \boldsymbol \theta)\in H^2(\Omega)$. Furthermore, the norm of a function $f\in H^2(\boldsymbol \theta(\Omega))$ is given by 
\beq \label{eq 4} \|f\|_{H^2(\boldsymbol \theta(\Omega))}=\frac{1}{c_{\boldsymbol \theta}} \left\{\sup_{0<r<1} \int_{\partial\Omega}|f\circ \boldsymbol\theta(r\z)|^2 |J_{\boldsymbol\theta}(r\z)|^2 d\mu(\z) \right\}^{1/2},\eeq where the normalization by $c_{\boldsymbol \theta}=\|J_{\boldsymbol\theta}\|_{H^2(\Omega)}$ ensures $\|1\|_{H^2(\boldsymbol \theta(\Omega))}=1.$
It follows from \cite[Proposition 3.13]{GR2024} that polynomials are dense in $H^2(\boldsymbol \theta(\Omega))$ and that it is a reproducing kernel Hilbert space with the reproducing kernel $\mathbb S_{\boldsymbol \theta(\Omega)}:\boldsymbol \theta(\Omega) \times \boldsymbol \theta(\Omega) \to \mathbb C$ given by 
\beq \label{Hardy kernel} \mathbb S_{\boldsymbol \theta(\Omega)}(\boldsymbol \theta(\z), \boldsymbol\theta(\boldsymbol w))=\frac{c_{\boldsymbol \theta}^2}{|G|}\frac{1}{J_{\boldsymbol \theta}(\z)\overline{J_{\boldsymbol \theta}(\boldsymbol w)}}  \sum_{\sigma\in G} \sgn(\sigma^{-1}) \mathbb S_{\Omega}(\sigma^{-1}\cdot \z, \boldsymbol w),\qquad \z, \boldsymbol w\in \Omega.\eeq

Let $L^2(\partial\boldsymbol \theta(\Omega))$ denote the space of square-integrable functions with respect to $\mu_{\boldsymbol\theta}.$ For $f\in L^2(\partial\boldsymbol \theta(\Omega)),$ the norm is given by 
\beq \label{eq 5} \|f\|_{L^2(\partial\boldsymbol \theta(\Omega))}= \frac{1}{c_{\boldsymbol \theta}} \|J_{\boldsymbol \theta} (f\circ \boldsymbol \theta)\|_{L^2(\partial \Omega)}= \frac{1}{c_{\boldsymbol \theta}}\left( \int_{\partial\Omega} |f(\boldsymbol \theta(\z))|^2 |J_{\boldsymbol \theta}(\z)|^2 d\mu(\z)\right)^{1/2}.\eeq

Recall that the proper map $\boldsymbol \theta=(\theta_1,\ldots, \theta_d) :\Omega \to \boldsymbol \theta(\Omega)$ consists of homogeneous polynomials; let $k_j$ denote the degree of $\theta_j$ for each $1\leq j\leq d.$  
If $f\in H^2(\boldsymbol \theta(\Omega))$, then $J_{\boldsymbol \theta}(f \circ \boldsymbol \theta) \in H^2(\Omega)$, and its boundary value $(J_{\boldsymbol \theta}(f \circ \boldsymbol \theta) )^*$ belongs to $ H^2(\partial\Omega).$ Since $J_{\boldsymbol \theta}$ is a homogeneous polynomial, the following limit exists for almost every $\z\in \partial\Omega$: \beq \label{eq 16}(J_{\boldsymbol \theta}(f \circ \boldsymbol \theta) )^*(z)=\lim_{r \to 1^-} J_{\boldsymbol \theta}(r\z)(f \circ \boldsymbol \theta)(r\z)=J_{\boldsymbol \theta}(\z) \lim_{r \to 1^-}f(r^{k_1} \theta_1(\z), \ldots, r^{k_d}\theta_d(\z)).\eeq
Let $V\subseteq \partial\Omega$ be the set of points where the limit in \eqref{eq 16} does not exist. Then the set $V$ is $G$-invariant and $\mu(V)=0$. Since the zero set $N(J_{\boldsymbol \theta})$ of $J_{\boldsymbol \theta}$ is $G$-invariant with $\mu(N(J_{\boldsymbol \theta}))=0$, we define a function $\tilde{f}^*$ on $\partial\boldsymbol \theta(\Omega)$ by
$$\tilde{f}^* (\boldsymbol \zeta)=\begin{cases}
    \frac{(J_{\boldsymbol \theta}(f \circ \boldsymbol \theta) )^*(z)}{J_{\boldsymbol \theta}(\z)}, \; \text{ if } \boldsymbol \zeta=\boldsymbol \theta(\z)\in \partial\boldsymbol \theta(\Omega) \text{ with } \z \notin V\cup N(J_{\boldsymbol \theta})\\
    0, \qquad\qquad\;\; \text{otherwise}
\end{cases}.$$
This gives a well defined element of $L^2(\partial\boldsymbol \theta(\Omega)).$ 
Given a function $f:\boldsymbol \theta(\Omega) \to \mathbb C$ and $0<r<1$, we define the slice $\tilde{f}_r:\partial\boldsymbol \theta(\Omega)\to \mathbb C$ by $$\tilde{f}_r(\boldsymbol \zeta)=f(r^{k_1}\zeta_1, r^{k_2} \zeta_2,\ldots, r^{k_d} \zeta_d),\qquad \boldsymbol \zeta=(\zeta_1,\ldots,\zeta_d)\in \partial\boldsymbol \theta(\Omega).$$ 
If $f\in H^2(\boldsymbol \theta(\Omega)),$ a routine verification shows that  $\tilde{f}_r \in L^2(\partial\boldsymbol \theta(\Omega))$ for all $0<r<1$, and that $\tilde{f}^*$ is measurable and square integrable with respect to $\mu_{\boldsymbol \theta}$, satisfying $$\lim_{r \to 1^-} \int_{\partial\boldsymbol \theta(\Omega)} |\tilde{f}_r-\tilde{f}^*|^2d\mu_{\boldsymbol \theta}=0.$$ Moreover, for all $\boldsymbol \zeta=\boldsymbol\theta(\z)$ with $\z\notin V\cup N(J_{\boldsymbol \theta})$, we have  $\tilde{ f}^*(\boldsymbol \zeta)=\lim_{r\to 1^-} \tilde{f}_r(\boldsymbol \zeta)$, and $J_{\boldsymbol \theta} (\tilde{f}^* \circ \boldsymbol \theta) \in H^2(\partial\Omega).$

Define $H^2(\partial\boldsymbol \theta(\Omega)):=\{f\in  L^2(\partial\boldsymbol \theta(\Omega)): J_{\boldsymbol \theta}(f \circ \boldsymbol \theta) \in H^2(\partial\Omega)\}.$ 
The following result establishes an isometric embedding of $H^2(\boldsymbol \theta(\Omega))$ onto the closed subspace $H^2(\partial\boldsymbol \theta(\Omega))$ of $L^2(\partial\boldsymbol \theta(\Omega)).$

\begin{proposition}\label{required result 1}
    Let $i': H^2(\boldsymbol \theta(\Omega)) \to L^2(\partial\boldsymbol {\theta}(\Omega))$ be the map given by $i'(f)=\tilde{f}^*.$ Then $i'$ is an isometry, and its range is precisely $H^2(\partial\boldsymbol \theta(\Omega)).$
\end{proposition}

\begin{proof} Let $\sgn(\cdot)$ denote the character of the sign representation of finite complex reflection group $G.$ A function $f$ in $d$ complex variables is said to be {\it $G$-alternating} if $\sigma\cdot f=\sgn(\sigma) f$ for all $\sigma\in G.$
  To relate the spaces $H^2(\boldsymbol \theta(\Omega))$ and $L^2(\partial\boldsymbol \theta(\Omega))$ to their corresponding function spaces on $\Omega$ and $\partial\Omega$, we consider the following subspaces:
\begin{align*}
    & H_{\sgn}^2(\Omega):=\left\{f\in H^2(\Omega): f \text{ is } G\text{-alternating}\right\}, \text{ and}\\
& L^2_{\sgn}(\partial\Omega):=\left\{h\in L^2(\partial\Omega): h \text{ is } G\text{-alternating }\mu \text{-a.e.} \right\}.
\end{align*} 
Clearly, the boundary value map $i: H^2(\Omega) \to H^2(\partial\Omega)$, when restricted to $H^2_{\sgn}(\Omega)$, surjects onto $H^2_{\sgn}(\partial\Omega):= L^2_{\sgn}(\partial\Omega) \cap H^2(\partial\Omega).$ 
 To transport this structure to the quotient domain $\boldsymbol \theta(\Omega)$, we consider the operators $\Gamma_h:H^2(\boldsymbol \theta(\Omega)) \to H^2_{\sgn}(\Omega)$ and  $\Gamma_{l}: L^2(\partial\boldsymbol \theta(\Omega))\to L_{\sgn}^2(\partial\Omega)$, defined by $$\Gamma_h f=\frac{1}{c_{\boldsymbol \theta}} J_{\boldsymbol \theta} (f \circ \boldsymbol \theta), \; f\in H^2(\boldsymbol \theta(\Omega)) \; \text{ and }\;  \Gamma_{l} g=\frac{1}{c_{\boldsymbol \theta}} J_{\boldsymbol\theta} (g \circ \boldsymbol\theta),\; g\in L^2(\partial\boldsymbol \theta(\Omega)).$$  
In view of \eqref{eq 4} and \eqref{eq 5}, both $\Gamma_h$ and $\Gamma_l$ are (well defined) isometries. In fact, they are unitary operators (see \cite[Lemma 3.5, Lemma 3.14]{GR2024}). This yields the following commutative diagram:
$$\begin{tikzcd}
H^2(\boldsymbol \theta(\Omega))\arrow{r}{\Gamma_l^{*} \circ i \circ \Gamma_h } \arrow[swap]{d}{\Gamma_h} &  L^2(\partial \boldsymbol \theta(\Omega)) \arrow{d}{\Gamma_l} \\
 H^2_{\sgn}(\Omega) \arrow{r}{i} &L^2_{\sgn}(\partial\Omega)
\end{tikzcd}.$$
 Hence, $\Gamma_l^{*} \circ i \circ \Gamma_h: H^2(\boldsymbol{\theta}(\Omega))\to L^2(\partial\boldsymbol \theta(\Omega))$ is an isometric embedding of $H^2(\boldsymbol \theta(\Omega))$ onto a closed subspace of $L^2(\partial\boldsymbol \theta(\Omega))$ (see \cite[Lemma 3.15]{GR2024}). It remains to verify that $i'=\Gamma_l^{*} \circ i \circ \Gamma_h$ and that its range is $H^2(\partial\boldsymbol \theta(\Omega)).$  
 
Let $f\in H^2(\boldsymbol \theta(\Omega)).$ By the preceding discussion, $\tilde{f}^*\in L^2(\partial\boldsymbol \theta(\Omega))$ and $J_{\boldsymbol \theta}(\tilde{f}^* \circ \theta)\in H_{\sgn}^2(\partial\Omega).$ Moreover, for almost every $\z\in \partial\Omega$, $$i(\Gamma_h(f))(\z)=\lim_{r \to 1^-} J_{\boldsymbol \theta}(r\z) f(\boldsymbol{\theta}(r\z))=J_{\boldsymbol \theta}(\z)\lim_{r\to 1^-} \tilde{f}_r(\boldsymbol{\theta}(\z)) =J_{\boldsymbol \theta}(\z) \tilde{f}^*(\boldsymbol{\theta}(\z)).$$
Consequently, $\Gamma_l^* \circ i \circ \Gamma_h (f)= \Gamma_l^*(J_{\boldsymbol{\theta}} \tilde{f}^* \circ \boldsymbol{\theta})=\tilde{f}^*=i'(f).$

To establish surjectivity, let $\tilde{f}^*\in H^2(\partial\boldsymbol \theta(\Omega)).$  
Equivalently,  $$\Gamma_l(\tilde{f}^*)=\frac{1}{c_{\boldsymbol \theta}} J_{\boldsymbol \theta} \tilde{f}^* \circ \boldsymbol \theta\in H^2_{\sgn}(\partial\Omega)=i(H^2_{\sgn}(\Omega)).$$ Hence, there exists $g\in H^2_{\sgn}(\Omega)$ such that $\frac{1}{c_{\boldsymbol \theta}} J_{\boldsymbol\theta} \tilde{f}^* \circ \boldsymbol\theta=g^*=i(g)\in H^2_{\sgn}(\partial\Omega).$ Setting $f=\Gamma_h^*(g)\in H^2(\boldsymbol \theta(\Omega)),$ a straightforward verification shows that $i'(f)=\tilde{f}^*.$ Therefore, $i'(H^2(\boldsymbol \theta(\Omega)))=H^2(\partial\boldsymbol \theta(\Omega)).$
\end{proof}

Let $L^\infty(\partial\boldsymbol \theta(\Omega))$ denote the algebra of essentially bounded $\mu_{\boldsymbol \theta}$-measurable functions on $\partial\boldsymbol \theta(\Omega)$, and let 
$L^\infty(\partial\Omega)^{G}=\{ f\in L^\infty (\partial\Omega): \sigma(f)=f\;  \mu\text{-a.e. for all } \sigma\in G\}.$ There exists an isometric $*$-isomorphism from $L^\infty(\partial\boldsymbol \theta(\Omega))$ onto $L^\infty(\partial\Omega)^{G}$ (see \cite[Subsection 3.6]{GR2024}). 
We define $$H^\infty(\partial \boldsymbol \theta(\Omega)):=H^2(\partial\boldsymbol \theta(\Omega)) \cap L^\infty(\partial\boldsymbol \theta(\Omega)).$$ Note that $H^\infty(\boldsymbol \theta(\Omega)) \subseteq H^2(\boldsymbol \theta(\Omega)).$ The following result shows that the boundary value map
$i'$ restricts to an isometric isomorphism between $H^\infty(\boldsymbol \theta(\Omega))$ and $H^\infty(\partial\boldsymbol \theta(\Omega)).$

\begin{proposition} \label{required result 2}
    For every $f\in H^\infty(\boldsymbol\theta(\Omega))$, the boundary value function $\tilde{f}^* $ belongs to $H^\infty(\partial\boldsymbol \theta(\Omega))$ and satisfies $\|\tilde{f}^*\|_{L^\infty(\partial\boldsymbol \theta(\Omega))}=\|f\|_{\infty, \boldsymbol\theta(\Omega)}.$ Furthermore, $i'(H^\infty(\boldsymbol \theta(\Omega)))=H^\infty(\partial\boldsymbol \theta(\Omega)).$ 
\end{proposition}
\begin{proof}
    Let $f\in H^\infty(\boldsymbol\theta(\Omega)).$ Then $\tilde{f}^* \in H^2(\partial\boldsymbol \theta(\Omega)).$ Since $|\tilde{f}_r(\boldsymbol \zeta)| \leq \|f\|_{\infty, \boldsymbol\theta(\Omega)}$ for all $0<r<1$ and $\boldsymbol \zeta \in \partial\boldsymbol \theta(\Omega),$ passing to the radial limits almost everywhere yields $\|\tilde{f}^*\|_{L^\infty(\partial\boldsymbol \theta(\Omega))} \leq \|f\|_{\infty, \boldsymbol \theta(\Omega)}.$ This confirms that $\tilde{f}^* \in H^\infty(\partial\boldsymbol \theta(\Omega)).$

    Conversely, suppose $h\in H^\infty (\partial \boldsymbol \theta(\Omega)).$ By Proposition \ref{required result 1}, there exists $f\in H^2(\boldsymbol \theta(\Omega))$ such that $i'(f)=h.$ In particular, $$\lim_{r \to 1^-} \tilde{f}_r(\boldsymbol \zeta)=h(\boldsymbol \zeta)\qquad \text{for almost every } \boldsymbol \zeta \in \partial\boldsymbol \theta(\Omega).$$ Consequently, for any polynomial $p$ in $d$-variables, $i'(pf)=p|_{\partial\boldsymbol \theta(\Omega)} h.$
    We show that $f$ is bounded and that its boundary function is $h.$
    Since $h \in L^\infty(\partial\boldsymbol \theta(\Omega))$, the multiplication operator $M_{h}: L^2(\partial\boldsymbol \theta(\Omega)) \to L^2(\partial\boldsymbol \theta(\Omega))$ defined by $M_{h}(g)= h g$ is bounded, with $\|M_{h}\|=\|h\|_{L^\infty(\partial\boldsymbol \theta(\Omega))}.$
    For every polynomial $p$, $$M_{h}(p|_{\partial\boldsymbol \theta(\Omega)})=p|_{\partial\boldsymbol \theta(\Omega)} h= i'(pf)\in H^2(\partial\boldsymbol\theta(\Omega)).$$ 
    As polynomials are dense in $H^2(\partial\boldsymbol \theta(\Omega))$, it follows that $M_{h}(H^2(\partial\boldsymbol \theta(\Omega))) \subseteq H^2(\partial\boldsymbol \theta(\Omega)).$
    Consequently, $T:=i'^{-1} \circ M_{h}\circ i'$ is a bounded operator on $H^2(\boldsymbol \theta(\Omega))$, and $\|T\| \leq \|h\|_{L^\infty(\partial\boldsymbol \theta(\Omega))}.$
    Moreover, for any polynomial $p$, $$fp=i'^{-1}( p|_{\partial\boldsymbol\theta(\Omega)}h)=i'^{-1}(M_{h} (i'(p)))=Tp\in H^2(\boldsymbol \theta(\Omega)).$$ 
    Hence, $f$ is a multiplier on the reproducing kernel Hilbert space $H^2(\boldsymbol \theta(\Omega))$, and the corresponding multiplier operator $\mathcal{M}_f$ is precisely $T$, ensuring $\|\mathcal{M}_f\| \leq \|h\|_{L^\infty(\partial\boldsymbol \theta(\Omega))}.$ 
    By the reproducing property, for any $\boldsymbol \zeta\in \boldsymbol \theta(\Omega)$, we have $\mathcal{M}_f^*(\mathbb S_{\boldsymbol \theta(\Omega)}(\cdot, \boldsymbol{\zeta}))=\overline{f(\boldsymbol \zeta)} \mathbb S_{\boldsymbol \theta(\Omega)}(\cdot, \boldsymbol{\zeta}),$ where $\mathbb S_{\boldsymbol \theta(\Omega)}(\cdot, \boldsymbol{\zeta})\in H^2(\boldsymbol \theta(\Omega))$ is the reproducing kernel at $\boldsymbol \zeta$, given by $\mathbb S_{\boldsymbol \theta(\Omega)}(\cdot, \boldsymbol{\zeta})(\boldsymbol \eta)=\mathbb S_{\boldsymbol \theta(\Omega)}(\boldsymbol \eta, \boldsymbol{\zeta})$ for all $\boldsymbol \eta\in \boldsymbol \theta(\Omega).$ 
    Therefore, $$|f(\boldsymbol \zeta)| \leq \frac{\|\mathcal{M}_f^*(\mathbb S_{\boldsymbol \theta(\Omega)}(\cdot, \boldsymbol{\zeta}))\|}{\|\mathbb S_{\boldsymbol \theta(\Omega)}(\cdot, \boldsymbol{\zeta})\|}\leq \|\mathcal{M}_f\| \leq \|h\|_{L^\infty(\partial\boldsymbol \theta(\Omega))}.$$ 
    Thus, $f$ is bounded and holomorphic on $\boldsymbol \theta (\Omega)$ with $\|f\|_{\infty, \boldsymbol \theta(\Omega)} \leq \|h\|_{L^\infty(\partial\boldsymbol \theta(\Omega))}.$
    On the other hand, the fact that $\lim_{r \to 1^-} \tilde{f}_r(\boldsymbol \zeta)=h(\boldsymbol \zeta)$  almost everywhere on $\boldsymbol \zeta \in \partial\boldsymbol \theta(\Omega)$ yields that $\|h\|_{L^\infty(\partial\boldsymbol \theta(\Omega))} \leq \|f\|_{\infty, \boldsymbol \theta(\Omega)}.$
    This establishes the isometric equality and completes the proof.
    \end{proof}

\section{Commutant lifting}\label{S3}
In this section, we characterize contractive module maps on quotient modules of \(H^2(\boldsymbol \theta(\Omega))\) that admit Schur-class lifts. 
More precisely, given a quotient module $\mathcal Q$ of $H^2(\boldsymbol \theta(\Omega))$, we classify all contractive module maps on $\mathcal Q$ that arise as compressions $S_{\varphi}=P_{\mathcal{Q}}T_{\varphi}|_{\mathcal{Q}}$ for some $\varphi\in \mathcal S(\boldsymbol \theta(\Omega))$.

By Proposition \ref{required result 1}, the Hardy space $H^2(\boldsymbol \theta(\Omega))$ is canonically identified with its boundary realization $H^2(\partial\boldsymbol \theta(\Omega)).$ 
Accordingly, we make no distinction between a function in $H^2(\boldsymbol \theta(\Omega))$ and its boundary realization unless necessary. Thus, every quotient module of $H^2(\boldsymbol \theta(\Omega))$ may be regarded as a closed subspace of $H^2(\partial\boldsymbol \theta(\Omega))$.
Henceforth, $\mathcal Q$ denotes a nonzero quotient module of $H^2(\partial\boldsymbol \theta(\Omega))$. 
Since $\mathcal{Q}^\perp$ is a submodule, the nontriviality of $\mathcal{Q}$ implies that $P_{\mathcal Q}1\neq 0.$
Consequently, \beq \label{Q description}\mathcal Q = \overline{\operatorname{span}}\left\{P_{\mathcal Q}\boldsymbol \zeta^\alpha P_{\mathcal Q}1: \alpha\in \mathbb Z_+^d\right\}.\eeq
For a closed subspace $\mathcal E\subseteq L^2(\partial\boldsymbol \theta(\Omega))$, set
$\mathcal E^{conj}=\{\bar{f}:\, f\in \mathcal E\}.$
Further, let $$H_0^2(\partial\boldsymbol \theta(\Omega))=\left\{f\in H^2(\partial\boldsymbol \theta(\Omega)):\langle f,1\rangle_{H^2(\partial\boldsymbol \theta(\Omega))}=0\right\}.$$
Under the boundary identification, this corresponds to the subspace of functions in $H^2(\boldsymbol \theta(\Omega))$ vanishing at the origin.
The following orthogonality relation will be used repeatedly.
\begin{lemma}\label{orthogonality}
    $H^2_0(\partial\boldsymbol \theta(\Omega))\perp H^2(\partial\boldsymbol \theta(\Omega))^{conj}.$
\end{lemma}
\begin{proof}
Let $\tilde f^*\in H_0^2(\partial\boldsymbol \theta(\Omega))$ and $\tilde g^*\in H^2(\partial\boldsymbol \theta(\Omega))$, and let $f\in H^2_0(\boldsymbol \theta(\Omega))$ and $g\in H^2(\boldsymbol \theta(\Omega))$ denote the associated interior functions. Since $\tilde f_r\tilde g_r \longrightarrow \tilde f^*\tilde g^*$ in $L^1(\partial\boldsymbol \theta(\Omega))$ as $r\to 1^-,$
it follows that \beq \label{eq 6}\langle \tilde f^*,\overline{\tilde g^*}\rangle_{L^2(\partial\boldsymbol \theta(\Omega))} = \lim_{r\to1^-} \frac{1}{c_{\boldsymbol \theta}^2}\int_{\partial\boldsymbol \theta(\Omega)} \tilde f_r\tilde g_r\, d\mu_{\boldsymbol \theta}=\frac{1}{c_{\boldsymbol \theta}^2}\lim_{r\to 1^-} \int_{\partial\Omega} |J_{\boldsymbol \theta}|^2 f(\boldsymbol \theta(r\z))g(\boldsymbol \theta(r\z)) d\mu(z).\eeq 
For a fixed $0<r<1$, define $H(\z)=f(\boldsymbol \theta(r\z))g(\boldsymbol \theta(r\z))$. Since the map $ \Omega\ni \z \mapsto H(\z)$ is holomorphic in a neighborhood of $\overline{\Omega},$ it admits a homogeneous expansion
 $H(\z)=\sum_{j=0}^\infty H_j(\z)$, where $H_0(\z)=f(\boldsymbol \theta(0)) g(\boldsymbol \theta(0)).$ Then $$\int_{\partial\Omega} H(\z)|J_{\boldsymbol \theta}|^2 d\mu(\z)= \sum_{j=0}^\infty \int_{\partial\Omega} H_j(\z) |J_{\boldsymbol \theta}(\z)|^2 d\mu(\z).$$
Since $\mu$ is rotation invariant and $J_{\boldsymbol \theta}$ is a homogeneous polynomial, $|J_{\boldsymbol \theta}(e^{it} \z)|=|J_{\boldsymbol \theta}(\z)|$ and $H_j(e^{it}\z)=e^{ijt} H_{j}(\z).$ 
Thus, we have $$\int_{\partial \Omega} H_{j}(\z)|J_{\boldsymbol \theta}(\z)|^2 d\mu(\z)=\begin{cases}
    c_{\boldsymbol\theta}^2 f(\boldsymbol\theta(0))g(\boldsymbol\theta(0)), \quad j=0\\
    0, \qquad\qquad\qquad\;\;\;\;\;\; j\neq 0
\end{cases}.$$
Consequently, $$\int_{\partial\Omega} H(\z)|J_{\boldsymbol \theta}|^2 d\mu(\z)=c_{\boldsymbol \theta}^2 f(\boldsymbol \theta(0)) g(\boldsymbol \theta(0)).$$
As $f(0)=0$ and $\boldsymbol \theta(0)=0$, \eqref{eq 6} yields $\langle \tilde f^*,\overline{\tilde g^*}\rangle=0.$
\end{proof}

It follows from Lemma \ref{orthogonality} that 
 $$ H^2(\partial\boldsymbol\theta(\Omega)) + H^2(\partial\boldsymbol\theta(\Omega))^{conj} = H_0^2(\partial\boldsymbol\theta(\Omega)) \oplus H^2(\partial\boldsymbol\theta(\Omega))^{conj}$$ is a closed subspace of $L^2(\partial\boldsymbol\theta(\Omega))$. 
Define $$\mathcal M_{\boldsymbol\theta(\Omega)} = L^2(\partial\boldsymbol\theta(\Omega)) \ominus \big( H^2(\partial\boldsymbol\theta(\Omega)) + H^2(\partial\boldsymbol\theta(\Omega))^{conj} \big).$$
By construction, the spaces $H^2_0(\partial\boldsymbol\theta(\Omega)),\;\mathcal M_{\boldsymbol\theta(\Omega)},$ and $\mathcal{Q}^{conj}$ are mutually orthogonal in $L^2(\partial\boldsymbol\theta(\Omega)).$ 
Accordingly, we define their algebraic direct sum in $L^1(\partial\boldsymbol\theta(\Omega))$: 
\beq \label{eq 14} \mathcal M_{\mathcal Q} =\mathcal Q^{conj} \dotplus \mathcal M_{\boldsymbol\theta(\Omega)} \dotplus H_0^2(\partial\boldsymbol\theta(\Omega)).\eeq 
We now establish our first characterization of commutant lifting on quotient modules of $H^2(\boldsymbol\theta(\Omega))$ (c.f \cite[Theorem 4.1]{DeS2026}).
\begin{theorem}\label{1st characterization}
    Let $\mathcal{Q}\subseteq  H^2(\partial\boldsymbol\theta(\Omega))$ be a quotient module and let $X\in \mathcal{B}_1(\mathcal{Q})$ be a contractive module map. Then $X$ admits a Schur-class lift if and only if the functional $X_{\mathcal{Q}}: \left(\mathcal{M}_{\mathcal{Q}}, \|\cdot\|_{L^1({\partial\boldsymbol\theta(\Omega)})}\right) \to \mathbb C$ defined by \beq \label{eq 11} X_{\mathcal{Q}}(f)=\int_{\partial\boldsymbol\theta(\Omega)}\psi f d\mu_{\boldsymbol\theta},\qquad f\in \mathcal{M}_{\mathcal{Q}}\eeq is contractive, where $\psi=XP_{\mathcal{Q}}1.$
\end{theorem}
\begin{proof}
    Assume first that $X\in \mathcal{B}_1(\mathcal{Q})$ admits a Schur-class lift. Then there exists $\varphi\in \mathcal{S}(\boldsymbol\theta(\Omega))$ such that $X=S_{\varphi}=P_{\mathcal{Q}} T_{\varphi}|_{\mathcal{Q}}.$
    Since $\varphi\in H^\infty(\boldsymbol\theta(\Omega))$ and $\mathcal{Q}^\perp$ is a submodule, it follows that $\varphi\mathcal{Q}^\perp \subseteq \mathcal{Q}^\perp.$ Hence, \beq \label{eq 7} P_{\mathcal{Q}}\varphi= P_{\mathcal{Q}}T_\varphi P_{\mathcal{Q}}1+ P_{\mathcal{Q}}T_\varphi P_{\mathcal{Q}^\perp}1=P_{\mathcal{Q}}T_{\varphi}P_{\mathcal{Q}}1=S_{\mathcal{\varphi}} P_{\mathcal{Q}}1=X P_{\mathcal{Q}}1=\psi.\eeq
Define the functional $\chi_{\varphi}: L^1(\partial\boldsymbol\theta(\Omega)) \to \mathbb C$ by 
$$\chi_{\varphi}(f)=\int_{\partial\boldsymbol\theta(\Omega)} \varphi f d\mu_{\boldsymbol\theta} = \inp{\varphi}{\bar{f}}_{L^2(\partial\boldsymbol\theta(\Omega))},\qquad f\in L^1(\partial\boldsymbol\theta(\Omega)).$$
Since $\varphi \in \mathcal{S}(\boldsymbol\theta(\Omega))$, $\|\chi_{\varphi}\| = \| \varphi\|_{L^\infty(\partial\boldsymbol\theta(\Omega))} \leq 1.$ We claim that $\chi_{\varphi}|_{\mathcal{M}_{\mathcal{Q}}}=X_{\mathcal{Q}}.$ Indeed, this follows by checking separately on the three summands of $\mathcal{M}_{\mathcal{Q}}$ in \eqref{eq 14}.
 
 If $\bar{f}\in \mathcal{Q}^{conj}$, then
\beqn X_{\mathcal{Q}}(\bar{f}) = \int_{\partial\boldsymbol\theta(\Omega)} \psi \bar{f} d\mu_{\boldsymbol \theta} = \inp{\psi}{f}_{H^2(\partial\boldsymbol\theta(\Omega))} \overset{\eqref{eq 7}}{=} \inp{P_{\mathcal{Q}} \varphi}{f}_{H^2(\partial\boldsymbol\theta(\Omega))}
= \inp{\varphi}{f}_{L^2(\partial\boldsymbol\theta(\Omega))} = \chi_{\varphi}(\bar{f}).\eeqn
Note that $\mathcal{M}_{\boldsymbol\theta(\Omega)} = \mathcal{M}_{\boldsymbol\theta(\Omega)}^{conj}$, so if $f\in \mathcal{M}_{\boldsymbol\theta(\Omega)}$, its conjugate $\bar{f}$ is also in $\mathcal{M}_{\boldsymbol\theta(\Omega)}$. Since $\varphi, \psi \in H^2(\partial\boldsymbol\theta(\Omega))$ are orthogonal to $\mathcal{M}_{\boldsymbol\theta(\Omega)}$, 
$$ X_{\mathcal{Q}}(f) = \inp{\psi}{\bar{f}}_{L^2(\partial\boldsymbol\theta(\Omega))} = 0 = \inp{\varphi}{\bar{f}}_{L^2(\partial\boldsymbol\theta(\Omega))} = \chi_{\varphi}(f) \qquad \text{for all} \; f\in \mathcal{M}_{\boldsymbol\theta(\Omega)}. $$
Finally, if $f\in H^2_{0}(\partial\boldsymbol\theta(\Omega))$,
then Lemma \ref{orthogonality} yields 
$$ X_{\mathcal{Q}}(f)  = \overline{\inp{\overline{\psi}}{f}_{L^2(\partial\boldsymbol\theta(\Omega))}} = 0 = \overline{\inp{\overline{\varphi}}{f}_{L^2(\partial\boldsymbol\theta(\Omega))}} = \chi_{\varphi}(f). $$
Hence, $X_{\mathcal{Q}} = \chi_{\varphi}|_{\mathcal{M}_{\mathcal{Q}}}$, and therefore $\|X_{\mathcal{Q}}\| \leq \|\chi_{\varphi}\| \leq 1.$

Conversely, suppose $X_{\mathcal{Q}}$ is contractive on $\mathcal{M}_{\mathcal{Q}}.$ By the Hahn-Banach theorem, there exists a bounded linear functional $\widetilde{X}_{\mathcal{Q}}: L^1(\partial\boldsymbol\theta(\Omega))\to \mathbb C$ such that $\widetilde{X}_{\mathcal{Q}}|_{\mathcal{M}_{\mathcal{Q}}} = X_{\mathcal{Q}}$ and $\|\widetilde{X}_{\mathcal{Q}}\| = \|X_{\mathcal{Q}}\| \leq 1.$ 
By duality, 
there exists $\varphi \in L^\infty(\partial\boldsymbol\theta(\Omega))$ such that $$\widetilde{X}_{\mathcal{Q}}(f) = \int_{\partial\boldsymbol\theta(\Omega)} \varphi f d\mu_{\boldsymbol \theta}=\inp{\varphi}{\bar{f}}_{L^2(\partial\boldsymbol\theta(\Omega))}, \qquad f \in L^1(\partial\boldsymbol\theta(\Omega)),$$
with $\|\varphi\|_{L^\infty(\partial\boldsymbol\theta(\Omega))} = \|\widetilde{X}_{\mathcal{Q}}\| \leq 1.$

We claim that $\varphi\in H^2(\partial\boldsymbol\theta(\Omega)).$ Since $$L^2(\partial\boldsymbol\theta(\Omega)) = H^2(\partial\boldsymbol\theta(\Omega)) \oplus H^2_0(\partial\boldsymbol\theta(\Omega))^{conj} \oplus \mathcal{M}_{\boldsymbol\theta(\Omega)},$$ it suffices to show $\varphi$ is orthogonal to the last two summands. If $f\in \mathcal{M}_{\boldsymbol\theta(\Omega)},$ then 
$$ \inp{\varphi}{f}_{L^2(\partial\boldsymbol\theta(\Omega))} = \widetilde{X}_{\mathcal{Q}}(\bar{f}) = X_{\mathcal{Q}}(\bar{f}) =  \inp{\psi}{f}_{L^2(\partial\boldsymbol\theta(\Omega))} = 0. $$
Similarly, if $\bar{h}\in H^2_0(\partial\boldsymbol\theta(\Omega))^{conj},$ then by Lemma \ref{orthogonality}
$$ \inp{\varphi}{\bar{h}}_{L^2(\partial\boldsymbol\theta(\Omega))} = \widetilde{X}_{\mathcal{Q}}(h) = X_{\mathcal{Q}}(h) = \overline{\inp{\overline{\psi}}{h}_{L^2(\partial\boldsymbol\theta(\Omega))}} = 0. $$
 Hence, $\varphi \in H^\infty(\partial\boldsymbol\theta(\Omega)).$ Since $\|\varphi\|_{L^\infty(\partial\boldsymbol\theta(\Omega))}\leq 1$, Proposition~\ref{required result 2} shows that $\varphi$ is the boundary realization of a Schur function on $\boldsymbol\theta(\Omega),$ which we continue to denote by $\varphi.$ 

Finally, we show that $X=S_{\varphi}.$ For any $f\in \mathcal{Q}$,
$$ \inp{\varphi}{f}_{H^2(\partial\boldsymbol\theta(\Omega))} = \widetilde{X}_{\mathcal{Q}}(\bar{f}) = X_{\mathcal{Q}}(\bar{f})= \inp{\psi}{f}_{H^2(\partial\boldsymbol\theta(\Omega))}. $$
Hence $P_{\mathcal{Q}} \varphi = \psi=XP_{\mathcal{Q}}1.$
Since $\mathcal{Q}^\perp$ is a submodule,
$$ XP_{\mathcal{Q}}1= P_{\mathcal{Q}}\varphi = P_{\mathcal{Q}} T_{\varphi}(P_{\mathcal{Q}}1 + P_{\mathcal{Q}^\perp}1) = P_{\mathcal{Q}} T_{\varphi}P_{\mathcal{Q}}1 = S_{\varphi} P_{\mathcal{Q}}1. $$
As $X$ commutes with compressed coordinate multipliers, $XS_{ \boldsymbol \zeta}^\alpha=S_{ \boldsymbol \zeta}^\alpha X$ for all $\alpha\in \mathbb Z_+^d.$ 
Therefore, 
\beqn
    S_{\varphi}(P_{\mathcal{Q}} \boldsymbol \zeta^\alpha P_{\mathcal{Q}}1) = S_{\boldsymbol \zeta}^\alpha S_{\varphi} P_{\mathcal{Q}}1 
    &= S_{\boldsymbol \zeta}^\alpha X P_{\mathcal{Q}}1 = X S_{\boldsymbol \zeta}^\alpha P_{\mathcal{Q}}1 = X (P_{\mathcal{Q}} \boldsymbol \zeta^\alpha P_{\mathcal{Q}}1).
\eeqn
In view of \eqref{Q description}, we conclude that $X=S_{\varphi}.$ This completes the proof of the theorem.
\end{proof}

\begin{remark}
The contractivity assumption on $X$ in Theorem \ref{1st characterization} can be relaxed to simple boundedness. More precisely, if $X\in \mathcal{B}(\mathcal{Q})$ is a bounded module map such that the associated functional $X_{\mathcal{Q}}$ on $\mathcal{M}_{\mathcal{Q}}$, defined by \eqref{eq 11}, is contractive, then $X$ admits a Schur-class lift. In particular, $X$ is then necessarily contractive.
\end{remark}

Building on Theorem~\ref{1st characterization}, we next reformulate the contractivity of the functional $X_{\mathcal{Q}}$ as a distance condition in $L^1(\partial\boldsymbol\theta(\Omega)).$ 
Recall that $X_{\mathcal{Q}}$ is completely determined by the vector $\psi=XP_{\mathcal{Q}}1 \in \mathcal{Q}.$ It follows from the proof of Theorem \ref{1st characterization} that $X_{\mathcal{Q}}$ vanishes on $\mathcal{M}_{\boldsymbol\theta(\Omega)}\dotplus H^2_0(\partial\boldsymbol\theta(\Omega))$ and on the orthogonal complement of $\overline{\psi}$ in $\mathcal{Q}^{conj}.$
Hence, the kernel of $X_{\mathcal{Q}}$ is given by
$$\widetilde{\mathcal{M}}_{\mathcal{Q}, X}:=\mathrm{ker}X_{\mathcal{Q}}=(\mathcal{Q}^{conj}\ominus \mathrm{span}\{\overline{\psi}\}) \dotplus \mathcal{M}_{\boldsymbol\theta(\Omega)}\dotplus H^2_{0}(\partial\boldsymbol\theta(\Omega)).$$ 
This leads to the following geometric characterization of commutant lifting. 
\begin{theorem}\label{2nd characterization}
    Let $\mathcal{Q}\subseteq H^2(\partial\boldsymbol\theta(\Omega))$ be a quotient module and let $X\in \mathcal{B}_1(\mathcal{Q})$ be a contractive module map. Set
$\psi=XP_{\mathcal Q}1.$ Then
    $X$ admits a Schur-class lift if and only if \beq \label{eq 15} \dist_{L^1(\partial\boldsymbol\theta(\Omega))}\left( \frac{\overline{\psi}}{\|\psi\|^2_{L^2(\partial\boldsymbol\theta(\Omega))}}, \widetilde{\mathcal{M}}_{\mathcal{Q}, X}\right) \geq 1.\eeq
\end{theorem}
\begin{proof}
    By the definition of the distance to a subspace, 
    $$ \dist_{L^1(\partial\boldsymbol\theta(\Omega))}\left( \frac{\overline{\psi}}{\|\psi\|^2_{L^2(\partial\boldsymbol\theta(\Omega))}}, \widetilde{\mathcal{M}}_{\mathcal{Q}, X}\right)=\inf \left\{\left\|\frac{\overline{\psi}}{\|\psi\|^2_{L^2(\partial\boldsymbol\theta(\Omega))}}+h\right\|_{L^1(\partial\boldsymbol\theta(\Omega))}: h\in \widetilde{\mathcal{M}}_{\mathcal{Q}, X} \right\}.$$ 
    Consequently, the distance condition in \eqref{eq 15} is equivalent to
   $$\|\lambda\overline{\psi}+h\|_{L^1(\partial\boldsymbol\theta(\Omega))} \geq |\lambda|\|\psi\|^2_{L^2(\partial\boldsymbol\theta(\Omega))}$$ for every $h \in \widetilde{\mathcal{M}}_{\mathcal{Q}, X}$ and every nonzero $\lambda \in \mathbb C.$
   
   Because $\mathcal{M}_{\mathcal{Q}} = \mathbb{C}\overline{\psi} \dotplus \widetilde{\mathcal{M}}_{\mathcal{Q}, X}$, every $f\in \mathcal{M}_{\mathcal Q}$ decomposes uniquely as $f = \lambda \overline{\psi} + h$ for some $\lambda\in\mathbb C$ and $h\in \widetilde{\mathcal{M}}_{\mathcal Q,X}$. Since $\mathrm{ker} X_{\mathcal{Q}}=\widetilde{\mathcal{M}}_{\mathcal Q,X}$,
\beq \label{eq 9} X_{\mathcal{Q}}(\lambda \overline{\psi} + h) = \lambda \int_{\partial\boldsymbol\theta(\Omega)} \psi \overline{\psi} d\mu_{\boldsymbol \theta} + X_{\mathcal{Q}}(h) = \lambda \|\psi\|^2_{L^2(\partial\boldsymbol\theta(\Omega))}. \eeq

If $X$ admits a Schur-class lift, Theorem \ref{1st characterization} gives that $X_{\mathcal{Q}}$ is contractive on $\mathcal{M}_{\mathcal{Q}}$. Thus, for every $\lambda \in \mathbb{C}$ and $h \in \widetilde{\mathcal{M}}_{\mathcal{Q}, X}$,
$$\|\lambda\overline{\psi} + h\|_{L^1(\partial\boldsymbol{\theta}(\Omega))} \geq |X_{\mathcal{Q}}(\lambda\overline{\psi} + h)| \overset{\eqref{eq 9}}{=} |\lambda| \|\psi\|^2_{L^2(\partial\boldsymbol{\theta}(\Omega))}.$$


Conversely, suppose that $\|\lambda\overline{\psi}+h\|_{L^1(\partial\boldsymbol\theta(\Omega))} \geq |\lambda|\|\psi\|^2_{L^2(\partial\boldsymbol\theta(\Omega))}$ for every nonzero $\lambda \in \mathbb C$ and $h\in \widetilde{\mathcal{M}}_{\mathcal{Q}, X}.$
Let $f = \lambda \overline{\psi} + h$ be an arbitrary element of $\mathcal{M}_{\mathcal{Q}}$. If $\lambda=0$, then $X_{\mathcal{Q}}(f) = 0 \leq \|f\|_{L^1(\partial\boldsymbol\theta(\Omega))}$. If $\lambda \neq 0$, then
$$ \|f\|_{L^1(\partial\boldsymbol\theta(\Omega))} = \|\lambda \overline{\psi} + h\|_{L^1(\partial\boldsymbol\theta(\Omega))} \geq |\lambda| \|\psi\|^2_{L^2(\partial\boldsymbol\theta(\Omega))}\overset{\eqref{eq 9}}{=} \left|X_{\mathcal{Q}}(\lambda \overline{\psi}+h)\right|.$$
Thus $X_{\mathcal Q}$ is contractive on $\mathcal{M}_{\mathcal Q}$,  and the conclusion follows from Theorem~\ref{1st characterization}.
\end{proof}

\subsection{Polydisc}
It was shown by Rudin \cite{Rd1983} that Schur functions on the Euclidean unit ball $\mathbb B_d$ can be approximated, in the weak-$*$ topology, by inner functions. More recently, this approximation phenomenon was used to obtain a characterization of commutant lifting on $\mathbb B_d$ in terms of inner functions (see \cite{BDS2025}). Motivated by these developments, we establish an analogous characterization of Schur-class liftability for quotient domains arising from the polydisc. We begin by fixing the relevant class of quotient domains and by formulating the corresponding notion of inner function.

Throughout this subsection, we take $\Omega=\mathbb D^d$, and let $G=G(m,n,d)$ be one of the imprimitive finite complex reflection groups in the Shephard-Todd classification \cite{ST1954}, where $m$ and $n$ are positive integers such that $n$ divides $m$. Let $\alpha$ be a primitive $m$-th  root of unity. The group $G(m,n,d)$ consists of transformations of the form 
$$\mathbb C^d \ni \z=(z_1,\ldots,z_d)\longmapsto (\alpha^{k_1} z_{\tau(1)}, \ldots, \alpha^{k_d} z_{\tau(d)}),$$ where $\tau\in \Sigma_d$ and $k_1,\ldots, k_d$ are integers such that $k_1+\cdots+k_d$ is divisible by $n.$

The associated basic polynomial map $\boldsymbol\theta=(\theta_1,\ldots,\theta_d):\mathbb D^d\to \boldsymbol \theta(\mathbb D^d)$ can be given explicitly: for each $1\leq j\leq d-1$, $\theta_j$ is the $j$-th elementary symmetric polynomial in $z_1^m,\ldots,z_d^m$, and $\theta_d(\boldsymbol z)=(z_1\cdots z_d)^{m/n}.$ 
The map $\boldsymbol\theta:\mathbb D^d\longrightarrow \boldsymbol\theta(\mathbb D^d)$ is proper and is factored by $G(m,n,d)$. In the special case $m=n=1$, the quotient domain $\boldsymbol\theta(\mathbb D^d)$ is the symmetrized polydisc. For notational simplicity, we shall write \(G\) for \(G(m,n,d)\) throughout the rest of this subsection.

Recall that the Shilov boundary of the quotient domain $\boldsymbol\theta(\mathbb D^d)$ is $\partial\boldsymbol \theta(\mathbb D^d)=\boldsymbol \theta(\mathbb T^d).$ In analogy with the polydisc, we say that a function $f\in H^\infty(\boldsymbol \theta(\mathbb D^d))$ is inner if the corresponding boundary function $\tilde{f}^*\in H^\infty(\boldsymbol  \theta(\mathbb T^d))$ is unimodular $\mu_{\boldsymbol \theta}$-a.e. on $\boldsymbol \theta(\mathbb T^d).$ 
Let $$\mathcal{I}(\boldsymbol \theta(\mathbb D^d))=\{f\in H^\infty(\boldsymbol \theta(\mathbb D^d)): f\circ\boldsymbol \theta \text{ is an inner function on } \mathbb D^d\}.$$ 
By Proposition \ref{required result 2}, for every $f\in H^\infty(\boldsymbol \theta(\mathbb D^d))$, $\tilde{f}^* \circ \boldsymbol \theta=(f \circ \boldsymbol \theta)^*$ almost everywhere on $\mathbb T^d$ and therefore the class of inner functions on $\boldsymbol \theta(\mathbb D^d)$ coincides with $\mathcal{I}(\boldsymbol \theta(\mathbb D^d)).$

We first prove that every Schur function on $\boldsymbol \theta(\mathbb D^d)$ is a weak-$*$ limit, in $L^\infty(\boldsymbol \theta(\mathbb T^d)),$ of inner functions on $\boldsymbol \theta(\mathbb D^d).$ 
To achieve this, we establish a preliminary result analogous to \cite[Theorem 5.5.1]{Rd1969}. The proof follows essentially the same line of argument as in Rudin’s theorem, with suitable modifications to incorporate the $G$-invariance condition. Recall that a function $f$ is $G$-invariant if $\sigma(f)=f$ for all $\sigma \in G.$
\begin{proposition}\label{caratheory}
    Let $f \in \mathcal{S}(\mathbb D^d)$ be $G$-invariant. Then there exists a sequence of $G$-invariant inner functions in the polydisc algebra $\mathcal{A}(\mathbb D^d)$ converging to $f$ uniformly on compact subsets of $\mathbb D^d.$ 
\end{proposition}
\begin{proof}
Fix a compact subset $K\subseteq \mathbb D^d$ and let $0<\epsilon<\frac{1}{2}.$ To enforce $G$-invariance, we define the orbit $\tilde{K}=\cup_{\sigma\in G} \sigma(K)$, which remains compact.
Choose a polynomial $P'$ such that $|P'(\z)| <1$ on $\overline{\mathbb D}^d$ and $|f(\boldsymbol{z}) - P'(\boldsymbol{z})| < \epsilon$ for all $\boldsymbol{z} \in \tilde{K}$.
To restore $G$-invariance, we consider the polynomial
$$ P(\boldsymbol{z}) = \frac{1}{|G|} \sum_{\sigma \in G} P'(\sigma(\z)). $$ Then $P$ is $G$-invariant and satisfies $|P(\z)| <1$  on $\overline{\mathbb D}^d.$ Moreover, for any $\z \in \tilde{K}$, we have
\beq \label{eq 17} |f(\boldsymbol{z}) - P(\boldsymbol{z})| \leq \frac{1}{|G|} \sum_{\sigma \in G} |f(\boldsymbol{z}) - P'(\sigma(\boldsymbol{z}))| = \frac{1}{|G|} \sum_{\sigma \in G} |f(\sigma(\boldsymbol{z}) - P'(\sigma(\boldsymbol{z})| < \epsilon, \eeq 
where we used the fact that $f$ is $G$-invariant and $\sigma(\z)\in \tilde{K}.$ For any polynomial $p$ in $d$ complex variables, let $\tilde{p}$ denote the polynomial defined by $\tilde{p}(\z)=\overline{p(\bar{\z})}.$  
Then, arguing as in \cite[Theorem~5.5.1]{Rd1969}, we obtain a monomial $M(\z)$ of the form $e^{i\nu} (z_1z_2\cdots z_d)^l$ of sufficiently high degree such that $m$ divides $l$ and satisfies $(i)$ $|M(\z)|=1$ on $\mathbb T^d$, $(ii)$ $M(\z) \tilde{P}(\frac{1}{\z})$ is a polynomial, and $(iii)$ both $|P(\z)M(\z)\tilde{P}(\frac{1}{\z})| <\epsilon$ and $|M(\z)|<\epsilon$ on $\tilde{K}.$ 
Define the rational function $$g(\z)=\frac{P(\z)+M(\z)}{1+M(\z)\tilde{P}(\frac{1}{\z})}.$$ Since $\alpha$ is $m$-th primitive root of unity and $l$ is divisible by $m$, then $\alpha^{l(k_1+\cdots+k_d)}=1$. Consequently, $M(\z)$ is $G$-invariant, that is, $M(\sigma(\z))=M(\z)$ for all $\sigma\in G.$ 
It is now straightforward to verify that $g$ is a $G$-invariant rational inner function in $\mathcal{A}(\mathbb D^d)$. Moreover, using \eqref{eq 17}, we obtain $|f(\z)-g(\z)| <5\epsilon$ for all $\z \in K.$ Since $K$ and $\epsilon$ were arbitrary, the result follows.
\end{proof}
 We now show that inner functions are weak-$*$ dense in the Schur class on $\boldsymbol \theta(\mathbb D^d).$

\begin{proposition}\label{weak*}
    Let  $\varphi\in \mathcal{S}(\boldsymbol \theta(\mathbb D^d)).$ Then there exists a sequence  $\{f_n\}\subseteq \mathcal{I}(\boldsymbol \theta(\mathbb D^d))$ converging to $\varphi$ in the weak-$*$ topology of $L^\infty(\boldsymbol \theta(\mathbb T^d)).$ 
\end{proposition}
\begin{proof}
    Let $\Phi=\varphi \circ \boldsymbol\theta.$ Then $\Phi$ is a $G$-invariant Schur function on $\mathbb D^d.$ By Proposition \ref{caratheory}, there exists a sequence $\{h_n\}$ of $G$-invariant inner functions in $\mathcal{A}(\mathbb D^d)$ such that $h_n \to \Phi$ uniformly on compact subsets of $\mathbb D^d.$ 
    By the analytic Chevalley-Shephard-Todd theorem (\cite[Theorem 3.1]{BDGR2022}), each $h_n$ factors through quotient map, that is, there exists $f_n\in H^\infty(\boldsymbol \theta(\mathbb D^d))$ such that $h_n =f_n \circ \boldsymbol \theta.$ 
    Since $h_n$ is inner, $f_n \in \mathcal{I}(\boldsymbol \theta(\mathbb D^d)).$
   
    We claim that $f_n\to \varphi$ in the weak-$*$ topology of $L^\infty(\boldsymbol \theta(\mathbb T^d)).$
    By the change of variables formula, it suffices to show that
    $$ \int_{\mathbb T^d} h_n u d\mu \to \int_{\mathbb T^d} \Phi u d\mu \qquad \text{for all}\; u \in L^1(\mathbb T^d).$$
    Since trigonometric polynomials are dense in $L^1(\mathbb T^d)$ and the functions $h_n$ are uniformly bounded, it is enough to verify that

    \beq \label{eq 13}  \int_{\mathbb{T}^d} h_n(\boldsymbol{z}) \boldsymbol{z}^\alpha \bar{\boldsymbol{z}}^\beta \, d\mu(\boldsymbol{z}) \to \int_{\mathbb{T}^d} \Phi(\boldsymbol{z}) \boldsymbol{z}^\alpha \bar{\boldsymbol{z}}^\beta \, d\mu(\boldsymbol{z}) \qquad \text{for all}\; \alpha , \beta \in \mathbb Z_+^d.\eeq
    The integrals in \eqref{eq 13} are partial derivatives of $h_n$ and $\Phi$ at the origin. The convergence of these integrals follows directly from the Cauchy integral formula and the uniform convergence of $h_n \to \Phi$ on compact subsets of $\mathbb{D}^d$.
    \end{proof} 
    
    The weak-$*$ density of inner functions established above yields the following approximation characterization for commutant lifting. 
    For a subset $S\subseteq L^1(\boldsymbol \theta(\mathbb T^d))$ and a function $\psi \in H^2(\boldsymbol \theta(\mathbb D^d))$, we say a sequence $\{f_n\}\subseteq \mathcal{S}(\boldsymbol \theta(\mathbb D^d))$ converges to $\psi$ in the weak-$*$ topology determined by $S$ if
$$\lim_{n\to\infty}\int_{\boldsymbol \theta(\mathbb T^d)} f_n g\, d\mu_{\boldsymbol\theta}=\int_{\boldsymbol \theta(\mathbb T^d)} \psi g\, d\mu_{\boldsymbol\theta}\qquad \text{for all } g\in S.$$    

\begin{theorem}\label{3rd characterization}
    Let  $\mathcal{Q}\subseteq H^2(\boldsymbol \theta(\mathbb T^d))$ be a quotient module and let $X\in \mathcal{B}_1(\mathcal{Q})$ be a contractive module map. Then, $X$ admits a lift if and only if there exists a sequence $\{f_n\}\subseteq \mathcal{I}(\boldsymbol \theta(\mathbb D^d))$ converging to $\psi=XP_{\mathcal{Q}}1$ in the weak-$*$ topology determined by $\mathcal{M}_{\mathcal{Q}}.$
\end{theorem}
\begin{proof}
    Assume first that $X\in \mathcal{B}_1(\mathcal{Q})$ admits a Schur-class lift. By Theorem \ref{1st characterization}, there exists $\varphi \in \mathcal{S}(\boldsymbol \theta(\mathbb D^d))$ such that the associated functional $X_{\mathcal{Q}}$ on $\mathcal{M}_{\mathcal{Q}}$ is given by integration against $\varphi.$ Equivalently, for every $g\in \mathcal{M}_{\mathcal{Q}}$, $$\int_{ \boldsymbol \theta(\mathbb T^d)} \psi g d\mu_{\boldsymbol \theta}=X_{\mathcal{Q}} g=\int_{\boldsymbol \theta(\mathbb T^d)} \varphi g d\mu_{\boldsymbol \theta}.$$   
    By Proposition \ref{weak*}, there exists a sequence $\{f_n\}\subseteq \mathcal{I}(\boldsymbol \theta(\mathbb D^d))$ converging to $ \varphi$ in the weak-$*$ topology of $L^\infty(\boldsymbol \theta(\mathbb T^d)).$ In particular, for every $g\in \mathcal{M}_{\mathcal{Q}},$
    $$\lim_{n\to\infty}\int_{\boldsymbol \theta(\mathbb T^d)} f_n g\, d\mu_{\boldsymbol \theta}
    =\int_{\boldsymbol \theta(\mathbb T^d)} \varphi g\, d\mu_{\boldsymbol \theta}
    =\int_{\boldsymbol \theta(\mathbb T^d)} \psi g\, d\mu_{\boldsymbol \theta}.$$
    Thus, $f_n \to \psi$ in the weak-$*$ topology determined by $\mathcal{M}_{\mathcal{Q}}.$
    
    Conversely, suppose there exists a sequence $\{f_n\}\subseteq \mathcal{I}(\boldsymbol \theta(\mathbb D^d))$ such that
    $$\lim_{n\to\infty}\int_{\boldsymbol \theta(\mathbb T^d)} f_n g\, d\mu_{\boldsymbol \theta}
    =\int_{\boldsymbol \theta(\mathbb T^d)} \psi g\, d\mu_{\boldsymbol \theta}
    \qquad \text{for all } g\in \mathcal{M}_{\mathcal{Q}}.$$
    Since each $f_n$ is inner, $|f_n|=1$ $\mu_{\boldsymbol \theta}$-a.e. on $\boldsymbol \theta(\mathbb T^d)$.
    Thus, for every $g\in \mathcal{M}_{\mathcal{Q}},$
   \beqn 
        |X_{\mathcal{Q}}g|
        = \left|\int_{\boldsymbol \theta(\mathbb T^d)} \psi g\, d\mu_{\boldsymbol \theta}\right|= \left|\lim_{n\to\infty}\int_{\boldsymbol \theta(\mathbb T^d)} f_n g \, d\mu_{\boldsymbol \theta}\right|
        \leq\|g\|_{L^1(\boldsymbol \theta(\mathbb T^d))}. 
    \eeqn
    Hence $X_{\mathcal Q}$ is contractive on $\mathcal M_{\mathcal Q}.$ The conclusion now follows from Theorem \ref{1st characterization}.
    \end{proof}
    
The weak-$*$ convergence condition in Theorem~\ref{3rd characterization} can be further simplified. 
Recall from the proof of Theorem \ref{1st characterization} that the associated functional $X_{\mathcal{Q}}$ vanishes on $\mathcal{M}_{\boldsymbol \theta(\mathbb D^d)}\dotplus H^2_0(\boldsymbol \theta(\mathbb T^d)).$  
Consequently, convergence against $\mathcal{M}_{\mathcal{Q}}$  reduces exclusively to convergence determined by $\mathcal{Q}^{conj}.$ 
This observation leads to the following immediate corollary. 
\begin{corollary}\label{refined characterization}
    Let $\mathcal{Q} \subseteq H^2(\boldsymbol \theta(\mathbb T^d))$ be a quotient module and let $X\in \mathcal{B}_1(\mathcal{Q})$ be a contractive module map. Then, $X$ admits a lift if and only if there exists a sequence $\{f_n\}\subseteq \mathcal{I}(\boldsymbol \theta(\mathbb D^d))$ converging to $\psi=XP_{\mathcal{Q}}1$ in the weak-$*$ topology determined by $\mathcal{Q}^{conj}.$
\end{corollary}

\section{Interpolation} \label{S4}
In this section we apply the commutant lifting results of Section~\ref{S3} to interpolation problems on quotient domains. 
This approach is motivated by Sarason’s operator-theoretic proof of the classical Nevanlinna–Pick interpolation theorem \cite[Section~4]{DS1967}, and by subsequent multivariable extensions for the polydisc and the Euclidean unit ball; see, for instance, \cite{DeS2026,BDS2025}.
Interpolation on the symmetrized bidisc, which is a proper image of the bidisc, was previously studied in \cite{AY2017,BS2018} using realization formulas for Schur functions. In contrast, our framework relies on commutant lifting and applies uniformly to quotient domains, including the symmetrized bidisc and the tetrablock.

The primary objective of this section is to characterize the existence of a function $\varphi\in \mathcal{S}(\boldsymbol \theta(\Omega))$ satisfying 
$$\varphi(\boldsymbol \zeta_i)=w_i, \qquad i=1,\ldots,p,$$
for a prescribed set of distinct points $\mathcal{Z}=\{\boldsymbol \zeta_1,\ldots,\boldsymbol \zeta_p\}\subseteq \boldsymbol \theta(\Omega)$ and target scalars $w_1,\ldots,w_p\in\mathbb D.$

We begin by recalling that $H^2(\boldsymbol \theta(\Omega))$ is a reproducing kernel Hilbert space with reproducing kernel $\mathbb S_{\boldsymbol \theta(\Omega)}$ is as given in \eqref{Hardy kernel}. For each $\zeta \in \boldsymbol \theta(\Omega)$, the kernel function $\mathbb S_{\boldsymbol \theta(\Omega)}(\cdot,\boldsymbol \zeta)$ satisfies
 the fundamental reproducing property
$$\inp{f}{\mathbb S_{\boldsymbol \theta(\Omega)}(\cdot,\boldsymbol \zeta)}_{H^2(\boldsymbol\theta(\Omega))}= f(\boldsymbol \zeta)\qquad \text{for all}\; f\in H^2(\boldsymbol \theta(\Omega)).$$
In particular, \beq \label{eq 10} T_{\zeta_i}^* \mathbb S_{\boldsymbol \theta(\Omega)}(\cdot, \boldsymbol \eta)=\bar{\eta}_i \mathbb S_{\boldsymbol\theta(\Omega)}(\cdot, \boldsymbol \eta), \qquad \boldsymbol \eta=(\eta_1,\ldots, \eta_d)\in \boldsymbol\theta(\Omega), \quad i=1,\ldots,d. \eeq
Given the set of $p$ distinct points $\mathcal{Z}$, we define the subspace
$$\mathcal{Q}_{\mathcal{Z}}=\mathrm{span}\{\mathbb S_{\boldsymbol\theta(\Omega)}(\cdot, \boldsymbol \zeta_i): i=1,\ldots, p\}.$$
It follows from \eqref{eq 10} that $\mathcal{Q}_{\mathcal{Z}}$ is a finite-dimensional quotient module of $H^2(\boldsymbol \theta(\Omega)).$ Its orthogonal complement is the space of functions vanishing on the interpolation nodes:
$$\mathcal{Q}_{\mathcal{Z}}^\perp =\left\{ g\in H^2(\boldsymbol \theta(\Omega)): g(\boldsymbol \zeta_i)=0 \text{ for all } i=1,\ldots,p\right\}.$$ We refer to quotient modules of this form as finite-dimensional {\it zero-based} quotient modules.
Since each kernel function $\mathbb S_{\boldsymbol \theta(\Omega)}(\cdot,\boldsymbol \zeta_j)$ is bounded, we have
 $\mathcal{Q}_{\mathcal{Z}} \subseteq H^\infty(\boldsymbol \theta(\Omega)).$

Let $X\in \mathcal{B}(\mathcal{Q}_{\mathcal{Z}})$ be a bounded module map,
and set $\psi =XP_{\mathcal{Q}_{\mathcal{Z}}} 1$. Then $ \psi \in \mathcal{Q}_{\mathcal{Z}}\subseteq H^\infty (\boldsymbol \theta(\Omega)).$
As $X$ is a module map, for any $\alpha \in \mathbb Z_+^d$ we obtain
\beqn X(P_{\mathcal{Q}_{\mathcal{Z}}} \boldsymbol \zeta^\alpha P_{\mathcal{Q}_{\mathcal{Z}}} 1) =X (S_{\boldsymbol \zeta}^\alpha  P_{\mathcal{Q}_{\mathcal{Z}}} 1)= S_{\boldsymbol \zeta}^\alpha X P_{\mathcal{Q}_{\mathcal{Z}}} 1= S_{\boldsymbol \zeta}^\alpha \psi=S_{\psi} P_{\mathcal{Q}_{\mathcal{Z}}} {\boldsymbol \zeta}^\alpha P_{\mathcal{Q}_{\mathcal{Z}}}1.\eeqn
Since $\mathcal{Q}_{\mathcal{Z}} =\overline{\mathrm{span}} \{P_{\mathcal{Q}_{\mathcal{Z}}} \boldsymbol \zeta^\alpha P_{\mathcal{Q}_{\mathcal{Z}}} 1: \alpha \in \mathbb Z_+^d \}$, it follows that $X=S_{\psi}$ on $\mathcal{Q}_{\mathcal{Z}}.$ 
Thus, every bounded module map on $\mathcal{Q}_{\mathcal{Z}}$ admits an $H^\infty(\boldsymbol \theta(\Omega))$-lift.
However, such a lift need not preserve the Schur-class norm structure on $\boldsymbol \theta(\Omega).$ 
Thus the interpolation problem reduces to deciding when the $H^\infty(\boldsymbol \theta(\Omega))$-lift can be chosen from the Schur class.

Using the commutant lifting framework developed in Section~\ref{S3}, we now establish the following interpolation theorem.
\begin{theorem}\label{interpolation 1st}
    Let $\mathcal{Z}=\{\boldsymbol \zeta_1,\ldots,\boldsymbol \zeta_p\}\subseteq \boldsymbol \theta(\Omega)$ be a set of $p$ distinct points, and let $w_1,\ldots,w_p\in \mathbb D.$
Define the scalars $c_1,\ldots,c_p$ via the Gram matrix inverse:
    $$\begin{bmatrix}
            c_1 \\  \vdots\\ c_p 
        \end{bmatrix}= {\begin{bmatrix}
            \mathbb S_{\boldsymbol \theta(\Omega)}(\boldsymbol \zeta_i, \boldsymbol \zeta_j)
        \end{bmatrix}}_{p\times p}^{-1} \begin{bmatrix}
            w_1 \\ \vdots \\w_p
        \end{bmatrix}.$$
    Set $\mathcal{M}_{\mathcal{Q}_{\mathcal{Z}}}=\mathcal{Q}_{\mathcal{Z}}^{conj} \dotplus \mathcal{M}_{\boldsymbol \theta(\Omega)}\dotplus H^2_{0}(\partial\boldsymbol \theta(\Omega))$ and $\psi_{\mathcal{Z}, \mathcal{W}}=\sum_{i=1}^p c_i \mathbb S_{\boldsymbol \theta(\Omega)}(\cdot, \boldsymbol \zeta_i). $ 
    Then there exists $\varphi\in \mathcal{S}(\boldsymbol \theta(\Omega))$ satisfying $\varphi(\boldsymbol \zeta_i)=w_i$ for all $i=1,\ldots,p$ if and only if the functional $X_{\mathcal{Q_{\mathcal{Z}}}}: \left(\mathcal{M}_{\mathcal{Q}_{\mathcal{Z}}}, \|\cdot \|_{L^1(\partial\boldsymbol \theta(\Omega))}\right) \to \mathbb C$ defined by $$X_{\mathcal{Q_{\mathcal{Z}}}} f=\int_{\partial\boldsymbol \theta(\Omega)} \psi_{\mathcal{Z}, \mathcal{W}} fd\mu_{\boldsymbol \theta}, \qquad f\in \mathcal{M}_{\mathcal{Q}_{\mathcal{Z}}}$$ is contractive.
     \end{theorem}
   
\begin{proof}
    Consider the operator $X_{\mathcal{Z,W}}\in \mathcal{B}(\mathcal{Q}_{\mathcal{Z}})$ defined by \beq \label{eq 12} X_{\mathcal{Z,W}}^* \mathbb S_{\boldsymbol \theta(\Omega)}(\cdot, \boldsymbol \zeta_i)=\overline{w}_i \mathbb S_{\boldsymbol \theta(\Omega)}(\cdot, \boldsymbol \zeta_i),\qquad i=1,\ldots,p.\eeq 
    A routine verification shows that $X_{\mathcal{Z},\mathcal{W}}$ is a bounded module map. Hence, by our preceding discussion, $\psi=X_{\mathcal{Z,W}}P_{\mathcal{Q}_{\mathcal{Z}}}1\in H^\infty(\boldsymbol \theta(\Omega))$ and $X_{\mathcal{Z,W}}=S_{\psi}$ on $\mathcal{Q}_{\mathcal{Z}}.$

    Suppose there exists $\varphi\in \mathcal{S}(\boldsymbol \theta(\Omega))$ such that $\varphi(\boldsymbol{\zeta}_i) = w_i$ for all $i = 1, \ldots, p$. Using the reproducing property,
    $$S_{\varphi}^* \mathbb S_{\boldsymbol \theta(\Omega)}(\cdot, \boldsymbol \zeta_i)=P_{\mathcal{Q}_{\mathcal{Z}}} T_{\varphi}^*|_{\mathcal{Q}_{\mathcal{Z}}} S_{\boldsymbol \theta(\Omega)}(\cdot, \boldsymbol \zeta_i) = \overline{\varphi(\boldsymbol \zeta_i)}\mathbb S_{\boldsymbol \theta(\Omega)}(\cdot, \boldsymbol \zeta_i)=\overline{w}_i \mathbb S_{\boldsymbol \theta(\Omega)}(\cdot, \boldsymbol \zeta_i).$$ 
    Comparing this with \eqref{eq 12}, we obtain $X_{\mathcal{Z,W}}=S_{\varphi}.$ Thus, $X_{\mathcal{Z},\mathcal{W}}$ admits a lift to the Schur function $\varphi.$
    
    Conversely, if $X_{\mathcal{Z,W}}=S_{\varphi}$ for some $\varphi\in \mathcal{S}(\boldsymbol \theta(\Omega)),$ then \eqref{eq 12} gives
    $$\overline{w}_i \mathbb S_{\boldsymbol \theta(\Omega)}(\cdot, \boldsymbol \zeta_i)=\overline{\varphi(\boldsymbol \zeta_i)} \mathbb S_{\boldsymbol \theta(\Omega)}(\cdot,\boldsymbol \zeta_i), \qquad i=1,\ldots,p,$$ and hence  $\varphi(\boldsymbol \zeta_i)=w_i$ for all $i=1,\ldots,p.$

    Therefore, the interpolation problem is solvable if and only if the operator $X_{\mathcal{Z},\mathcal{W}}$ admits a lift to the Schur-class interpolant.
    By Theorem \ref{1st characterization}, this occurs if and only if the functional 
    $X_{\mathcal{Q}_{\mathcal{Z}}}: \left(\mathcal{M}_{\mathcal{Q}_{\mathcal{Z}}}, \|\cdot \|_{L^1(\partial\boldsymbol \theta(\Omega))}\right) \to \mathbb C$ defined by $$X_{\mathcal{Q}_{\mathcal{Z}}}(f)=\int_{\partial\boldsymbol \theta(\Omega)} \psi f d\mu_{\boldsymbol \theta}, \qquad f\in \mathcal{M}_{\mathcal{Q}_{\mathcal{Z}}},$$
    is contractive.
    Moreover, in this contractive setting, $X_{\mathcal{Z}, \mathcal{W}} = S_{\psi} = S_{\varphi}$ on $\mathcal{Q}_{\mathcal{Z}}$, and $\psi(\boldsymbol{\zeta}_i) = w_i$ for all $i = 1, \ldots, p.$
    
    It remains only to identify $\psi.$ Since $\psi\in \mathcal{Q}_{\mathcal{Z}},$ there are scalars $c_1,\ldots,c_p$ such that $\psi=\sum_{j=1}^p c_j \mathbb S_{\boldsymbol \theta(\Omega)}(\cdot, \boldsymbol \zeta_j).$ Evaluating at each node $\boldsymbol \zeta_i$ yields
    $$ w_i=\psi(\boldsymbol \zeta_i)= \sum_{j=1}^p c_j \mathbb S_{\boldsymbol \theta(\Omega)}(\boldsymbol \zeta_i, \boldsymbol \zeta_j),\qquad i=1,\ldots,p.$$
    This uniquely determines the coefficients $c_1,\ldots, c_p$, and hence  $\psi=\psi_{\mathcal{Z,W}}.$
    \end{proof}
    
We next derive a geometric formulation of the interpolation problem using the distance characterization of commutant lifting.
As shown in the proof of Theorem~\ref{interpolation 1st}, the existence of a function $\varphi \in \mathcal{S}(\boldsymbol \theta(\Omega))$ satisfying $\varphi(\boldsymbol \zeta_i)=w_i$ for all $i=1,\ldots, p$ 
 is equivalent to the fact that the operator $X_{\mathcal{Z,W}}$ defined by \eqref{eq 12} admits a lift to $\varphi.$ In particular, such a lift necessarily implies that $\|X_{\mathcal{Z,W}}\|\leq 1.$ Applying Theorem \ref{2nd characterization} to $X_{\mathcal{Z,W}}$ therefore yields the following interpolation criterion.


\begin{theorem}\label{interpolation 2nd}
     Let $\mathcal{Z}=\{\boldsymbol \zeta_1,\ldots,\boldsymbol \zeta_p\}\subseteq \boldsymbol \theta(\Omega)$ be a set of $p$ distinct points, and let $w_1,\ldots,w_p\in \mathbb D$. Set $\psi=X_{\mathcal{Z,W}} P_{\mathcal{Q}_{\mathcal{Z}}}1$ and define $$\widetilde{\mathcal M}_{\mathcal{Q}_{\mathcal{Z}}}=\left(\mathcal{Q}_{\mathcal{Z}}^{conj}\ominus\mathrm{span}\{\overline{\psi}\}\right)\dotplus \mathcal{M}_{\boldsymbol \theta(\Omega)}\dotplus H^2_0(\partial\boldsymbol \theta(\Omega)).$$ Then there exists $\varphi\in \mathcal{S}(\boldsymbol \theta(\Omega))$ such that $\varphi(\boldsymbol \zeta_i)=w_i$ for all $i=1,\ldots, p$ if and only if 
     $$\dist_{L^1(\partial\boldsymbol \theta(\Omega))}\left(\frac{\overline{\psi}}{\|\psi\|^2_{L^2(\partial\boldsymbol \theta(\Omega))}}, \; \widetilde{\mathcal M}_{\mathcal{Q}_{\mathcal{Z}}}\right)\geq 1.$$
\end{theorem}

We now pass to quotient domains of the polydisc associated with the finite complex reflection groups $G(m,n,d)$. For these domains, Theorem \ref{3rd characterization} gives a commutant lifting criterion in terms of inner functions. Applying that criterion to the finite-dimensional zero-based quotient modules determined by the interpolation nodes yields the following inner-function characterization of Schur-class interpolation.

\begin{theorem}\label{interpolation 3rd}
    Let $\mathcal{Z}=\{\boldsymbol \zeta_1,\ldots,\boldsymbol \zeta_p\}\subseteq \boldsymbol \theta(\mathbb D^d)$ be a set of $p$ distinct points, and let $w_1,\ldots,w_p\in \mathbb D.$ Then there exists $\varphi \in \mathcal{S}(\boldsymbol \theta(\mathbb D^d))$ such that $\varphi(\boldsymbol \zeta_i)=w_i$ for all $i=1,\ldots, p$ if and only if there exists a sequence $\{f_n\}\subseteq \mathcal{I}(\boldsymbol \theta(\mathbb D^d))$ satisfying $$\lim_{n \to \infty} f_n(\boldsymbol \zeta_j)=w_j, \qquad j=1,\ldots, p.$$
\end{theorem}
\begin{proof}
Let $\mathcal{Q}_{\mathcal{Z}}$ be the zero-based quotient module associated with $\mathcal{Z}$, and let $X_{\mathcal{Z,W}}\in \mathcal{B}(\mathcal{Q}_{\mathcal{Z}})$ be the module map defined by \eqref{eq 12}. Set $\psi=X_{\mathcal{Z,W}}P_{\mathcal{Q}_{\mathcal{Z}}}1.$ As in the proof of Theorem \ref{interpolation 1st}, the existence of an interpolating Schur function $\varphi$ is equivalent to $X_{\mathcal{Z,W}}$ admitting a lift to $\varphi.$ 
In this scenario, $$X_{\mathcal{Z,W}}=S_{\varphi}=S_{\psi} \quad\text{and}\quad \psi(\boldsymbol \zeta_i)=w_i \;\text{for all}\; i=1,\ldots,p.$$

Assume first that such a function $\varphi$ exists. By Theorem~\ref{3rd characterization}, there exists a sequence $\{f_n\}\subseteq \mathcal{I}(\boldsymbol \theta(\mathbb D^d))$ converging to $\psi$ in the weak-$*$ topology determined by $\mathcal{M}_{\mathcal{Q}_{\mathcal{Z}}}$.
In particular, $$\int_{\boldsymbol \theta(\mathbb T^d)} f_n \overline{\mathbb S_{\boldsymbol \theta(\mathbb D^d)}(\cdot, \boldsymbol \zeta_j)} d\mu_{\boldsymbol \theta} \to \int_{\boldsymbol \theta(\mathbb T^d)} \psi \overline{\mathbb S_{\boldsymbol \theta(\mathbb D^d)}(\cdot, \boldsymbol \zeta_j)} d\mu_{\boldsymbol \theta}, \qquad j=1,\ldots, p.$$
Applying the reproducing property, this yields
$$f_n(\boldsymbol \zeta_j)\to \psi(\boldsymbol \zeta_j)=w_j, \qquad j=1,\ldots,p.$$

Conversely, suppose that there exists a sequence $\{f_n\}\subseteq \mathcal{I}(\boldsymbol \theta(\mathbb D^d))$ such that $$ \lim_{n\to\infty} f_n(\boldsymbol{\zeta}_j)=w_j, \qquad j=1,\ldots,p.$$ 
As established prior to Theorem \ref{interpolation 1st}, the module map $X_{\mathcal{Z,W}}$ coincides with $S_{\psi}$, where $\psi= X_{\mathcal{Z,W}}P_{\mathcal{Q}_{\mathcal{Z}}}1.$ Moreover, \eqref{eq 12} and the reproducing property imply that $$\psi(\boldsymbol \zeta_j)=w_j, \qquad j=1,\ldots, p.$$
Since $\mathcal Q_{\mathcal Z}$ is spanned by the kernel functions at
$\boldsymbol\zeta_1,\ldots,\boldsymbol\zeta_p$, the reproducing property shows that $f_n$ converges to $\psi$ in the weak-$*$ topology determined by
$\mathcal Q_{\mathcal Z}^{\operatorname{conj}}$. By Corollary \ref{refined characterization}, $X_{\mathcal Z,\mathcal W}$ admits a Schur-class lift.
Hence, there exists $\varphi\in\mathcal S(\boldsymbol\theta(\mathbb D^d))$ such that $X_{\mathcal{Z,W}}=S_{\varphi}.$
Using \eqref{eq 12} and the reproducing property once more, we conclude that
    $$\varphi(\boldsymbol \zeta_j)=w_j, \qquad j=1,\ldots,p.$$ This completes the proof.
\end{proof}

  \subsubsection*{Acknowledgment} The author expresses his sincere gratitude to his PhD supervisor, Professor Surjit Kumar, for many valuable suggestions and comments during the preparation of this draft.

\end{document}